\documentclass[preprint,12pt]{elsarticle}
\usepackage{amssymb}
\usepackage{amsthm}
\usepackage{amsmath}
\usepackage{amsfonts}
\usepackage{xcolor}

\newtheorem{theorem}{Theorem}

\newtheorem{corollary}{Corollary}

\newtheorem{definition}{Definition}
\newtheorem{remark}{Remark}
\usepackage{ragged2e}

\usepackage[colorlinks=true]{hyperref}%
\journal{arXiv}

\begin{document}
\hypersetup{
  colorlinks,
  citecolor=red,
  linkcolor=red,
  urlcolor=blue}
  
  \hypersetup{
  citebordercolor=red,
  filebordercolor=red,
  linkbordercolor=blue
}
\begin{frontmatter}


 \tnotetext[label1]{\footnotesize {corresponding author\\Dagnachew Jenber}}

\title{Type of Leibniz Rule on Riemann-Liouville Variable-Order Fractional Integral and Derivative Operator}


\author{Dagnachew Jenber$^{a,\star}$}
\address[label]{Department of Mathematics, Addis Ababa Science and Technology University, Addis Ababa, Ethiopia\\
Department of Mathematics, Bahir Dar University, Bahir Dar, Ethiopia, P.O.Box 79\\ Email: djdm$\_$101979@yahoo.com}


\author{Mollalign Haille$^{b}$}
\address[label]{Department of Mathematics, Bahir Dar University, Bahir Dar, Ethiopia, P.O.Box 79\\ Email: mollalgnhailef@gmail.com}

\begin{abstract}
In this paper, types of Leibniz Rule for Riemann-Liouville Variable-Order fractional integral and derivative Operator is developed. The product rule, quotient rule, and chain rule formulas for both integral and differential operators are established. In particular, there are four types of product rule formulas: Product rule type-I, Product rule type-II, Product rule type-III and Product rule type-Iv. Quotient rule type-I, quotient rule type-II, quotient rule type-III, and quotient rule type-Iv formulas developed from product rule types. There are four types of chain rule formulas: chain rule type-I, chain rule type-II, chain rule type-III, and chain rule type-Iv. 
\end{abstract}


\end{frontmatter}


\section{Introduction}
\label{S:1}
Fractional calculus, that is fractional derivative and integral of an arbitrary real order, has a history of more than three hundred years (see \cite{1},\cite{2} and the references therein). In 1993, Samko and Ross \cite{3} firstly proposed the notion of variable-order integral and
differential operators and some basic properties. Lorenzo and Hartley
\cite{4} summarized the research results of the variable-order fractional operators and
then investigated the definitions of variable-order fractional operators in different
forms. After that, some new extensions and valuable application potentials
of the variable-order fractional differential equation models have been further explored \cite{5}. It has become a research hotspot and has aroused wide concern in the last ten years.
Different kind of definitions of fractional derivatives and integrals are available in the literature. Forexample, Riemann-Liouville, Riesz, Caputo, Coimbra, Hadamard, Grünwald-Letnikov, Marchaud, Weyl, Sonin-Letnikov, conformable and others (see \cite{6},\cite{7}, \cite{15} and the references therein). Excepting conformable fractional derivative (see \cite{9}) the other definition violates basic properties of Leibniz rule that holds for integer order calculus, like product rule and chain rule. V.E. Tarasov proved that fractional derivatives of non-integer orders can not satisfy the Leibniz rule (see \cite{13},\cite{14}). There are some attempts to define new type of fractional derivative such that the Leibniz rule holds (see \cite{10},\cite{11},\cite{12}). This paper established a Leibnize rule type formula like product rule, quotient rule and chain rule for Riemann-Liouville variable-order fractional derivative and integral operator. We will leave linearity property for the reader to check, since it is obvious and straightforward.

\section{Preliminaries}
\label{S:2}
Throughout this paper, we will use the following definitions.

\begin{definition} Given $\Re(z)>0$, we define the gamma function, $\Gamma(z)$, as
\begin{equation*}
\Gamma(z)=\int_{0}^{\infty}t^{z-1}e^{-t}dt
\end{equation*}
$\Gamma(z)$ is a holomorphic function in $\Re(z)>0$.
\end{definition}
\justify
In the following definition of Riemann-Liouville variable-order fractional integral, we used the abbreviation RL stands for Riemann-Liouville. 
\begin{definition}(see\cite{8}) Let $\alpha : [a,b] \times [a,b]\longrightarrow (0,\infty)$. Then the left Riemann-Liouville fractional integral of order $\alpha(.,.)$ for function $f(t)$ is defined by
\begin{equation}
\label{2.2}
{}_{a}^{RL}I_{t}^{\alpha(.,.)}f(t)=\int_{a}^{t}\frac{(t-s)^{\alpha(t,s)-1}}{\Gamma(\alpha(t,s))}f(s)ds,\text{ }t>a
\end{equation}
\end{definition}

\begin{definition}(see\cite{8}) Let $\alpha : [a,b] \times [a,b]\longrightarrow (0,1)$. Then the left Riemann-Liouville fractional derivative of order $\alpha(.,.)$ for function $f(t)$ is defined by
\begin{equation}
\label{2.3}
{}_{a}^{RL}D_{t}^{\alpha(.,.)}f(t)=\frac{d}{dt} \bigg({}_{a}^{RL}I_{t}^{1-\alpha(.,.)}f(t)\bigg)=\frac{d}{dt} \int_{a}^{t}\frac{(t-s)^{-\alpha(t,s)}}{\Gamma(1-\alpha(t,s))}f(s)ds,\text{ }t>a
\end{equation}
\end{definition}



\section{Main Result}
\label{S:4}




For the Reimann-Liouville variable-order fractional integral operator, from Theorem~(\ref{4.4}), we get, product rule formulas and from the consequence of this Theorem, product rule type-I, product rule type-II, product rule type-III and product rule type-IV are obtained.
\begin{theorem}
\label{4.4}
Let $\alpha,\beta : [a,b] \times [a,b]\longrightarrow (0,\infty)$, $a,c\in \mathbb{R}$, $t>a,s>c$. Then for functions $f$ and $g$ the following equality holds
\begin{equation}
\label{4.3}
\begin{split}
\bigg({}_{a}^{RL}I_{t}^{\alpha(.,.)}(fg)(t)\bigg) \bigg({}_{c}^{RL}I_{s}^{\beta(.,.)}(1) \bigg)+\bigg({}_{a}^{RL}I_{t}^{\alpha(.,.)}(1)\bigg)\bigg( {}_{c}^{RL}I_{s}^{\beta(.,.)}(fg)(s)\bigg)
\\
=\bigg({}_{c}^{RL}I_{s}^{\beta(.,.)}\bigg({}_{a}^{RL}I_{t}^{\alpha(.,.)}(f(t)-f(s))(g(t)-g(s))\bigg)\bigg)
+\bigg( {}_{a}^{RL}I_{t}^{\alpha(.,.)}g(t)\bigg)\\
\times \bigg( {}_{c}^{RL}I_{s}^{\beta(.,.)}f(s)\bigg)+\bigg( {}_{a}^{RL}I_{t}^{\alpha(.,.)}f(t)\bigg)\bigg( {}_{c}^{RL}I_{s}^{\beta(.,.)}g(s)\bigg)
\end{split}
\end{equation}
\end{theorem}
\begin{proof}
Since
\begin{equation}
\label{4.1}
f(x)g(x)=(f(x)-f(y))(g(x)-g(y))+f(y)g(x)+f(x)g(y)-f(y)g(y).
\end{equation}
Now, multiplying equation~(\ref{4.1}) by $(t-x)^{\alpha(t,x)-1}/\Gamma(\alpha(t,x))$ and integrate from $a$ to $t$ with respect to $x$, we have
\begin{equation*}
\begin{split}
&\int_{a}^{t}\frac{(t-x)^{\alpha(t,x)-1}}{\Gamma(\alpha(t,x))} f(x)g(x)dx\\
&\quad=\int_{a}^{t}\frac{(t-x)^{\alpha(t,x)-1}}{\Gamma(\alpha(t,x))}(f(x)-f(y))(g(x)-g(y))dx
\\
&\qquad+\int_{a}^{t}\frac{(t-x)^{\alpha(t,x)-1}}{\Gamma(\alpha(t,x))}f(y)g(x)dx+\int_{a}^{t}\frac{(t-x)^{\alpha(t,x)-1}}{\Gamma(\alpha(t,x))}f(x)g(y)dx
\\
&\qquad-\int_{a}^{t}\frac{(t-x)^{\alpha(t,x)-1}}{\Gamma(\alpha(t,x))}f(y)g(y)dx
\end{split}
\end{equation*}
which means

\begin{equation}
\label{4.2}
\begin{split}
{}_{a}^{RL}I_{t}^{\alpha(.,.)}(fg)(t)=&{}_{a}^{RL}I_{t}^{\alpha(.,.)}\bigg((f(t)-f(y))(g(t)-g(y))\bigg)
+f(y)\bigg( {}_{a}^{RL}I_{t}^{\alpha(.,.)}g(t)\bigg)\\
&\quad+g(y)\bigg({}_{a}^{RL}I_{t}^{\alpha(.,.)}f(t)\bigg)-\bigg(f(y)g(y)\bigg) \bigg({}_{a}^{RL}I_{t}^{\alpha(.,.)}(1)\bigg)
\end{split}
\end{equation}
Now, multiplying equation~(\ref{4.2}) by $(s-y)^{\beta(s,y)-1}/\Gamma(\beta(s,y))$ and integrate from $c$ to $s$ with respect to $y$, we get,
\begin{equation*}
\begin{split}
&\int_{c}^{s}\frac{(s-y)^{\beta(s,y)-1}}{\Gamma(\beta(s,y))} {}_{a}^{RL}I_{t}^{\alpha(.,.)}(fg)(t)dy\\
&\quad=\int_{c}^{s}\frac{(s-y)^{\beta(s,y)-1}}{\Gamma(\beta(s,y))}\bigg({}_{a}^{RL}I_{t}^{\alpha(.,.)}(f(t)-f(y))(g(t)-g(y))\bigg)dy
\\
&\qquad+\int_{c}^{s}\frac{(s-y)^{\beta(s,y)-1}}{\Gamma(\beta(s,y))}f(y)\bigg( {}_{a}^{RL}I_{t}^{\alpha(.,.)}g(t)\bigg)dy\\
&\qquad+\int_{c}^{s}\frac{(s-y)^{\beta(s,y)-1}}{\Gamma(\beta(s,y))}g(y)\bigg({}_{a}^{RL}I_{t}^{\alpha(.,.)}f(t)\bigg)dy
\\
&\qquad-\int_{c}^{s}\frac{(s-y)^{\beta(s,y)-1}}{\Gamma(\beta(s,y))}\bigg(f(y)g(y)\bigg) \bigg({}_{a}^{RL}I_{t}^{\alpha(.,.)}(1)\bigg)dy
\end{split}
\end{equation*}
which means

\begin{equation*}
\begin{split}
&\bigg({}_{a}^{RL}I_{t}^{\alpha(.,.)}(fg)(t)\bigg) \bigg({}_{c}^{RL}I_{s}^{\beta(.,.)}(1) \bigg)\\
&\quad=\bigg({}_{c}^{RL}I_{s}^{\beta(.,.)}\bigg({}_{a}^{RL}I_{t}^{\alpha(.,.)}(f(t)-f(s))(g(t)-g(s))\bigg)\bigg)
\\
&\qquad+\bigg( {}_{a}^{RL}I_{t}^{\alpha(.,.)}g(t)\bigg)\bigg( {}_{c}^{RL}I_{s}^{\beta(.,.)}f(s)\bigg)+\bigg( {}_{a}^{RL}I_{t}^{\alpha(.,.)}f(t)\bigg)\bigg( {}_{c}^{RL}I_{s}^{\beta(.,.)}g(s)\bigg)
\\
&\qquad-\bigg({}_{a}^{RL}I_{t}^{\alpha(.,.)}(1)\bigg)\bigg( {}_{c}^{RL}I_{s}^{\beta(.,.)}(fg)(s)\bigg)
\end{split}
\end{equation*}
which means

\begin{equation*}
\begin{split}
\bigg({}_{a}^{RL}I_{t}^{\alpha(.,.)}(fg)(t)\bigg) \bigg({}_{c}^{RL}I_{s}^{\beta(.,.)}(1) \bigg)+\bigg({}_{a}^{RL}I_{t}^{\alpha(.,.)}(1)\bigg)\bigg( {}_{c}^{RL}I_{s}^{\beta(.,.)}(fg)(s)\bigg)
\\
=\bigg({}_{c}^{RL}I_{s}^{\beta(.,.)}\bigg({}_{a}^{RL}I_{t}^{\alpha(.,.)}(f(t)-f(s))(g(t)-g(s))\bigg)\bigg)
+\bigg( {}_{a}^{RL}I_{t}^{\alpha(.,.)}g(t)\bigg)
\\
\times \bigg( {}_{c}^{RL}I_{s}^{\beta(.,.)}f(s)\bigg)+\bigg( {}_{a}^{RL}I_{t}^{\alpha(.,.)}f(t)\bigg)\bigg( {}_{c}^{RL}I_{s}^{\beta(.,.)}g(s)\bigg)
\end{split}
\end{equation*}

\end{proof}

From Theorem~(\ref{4.4}), we established the following 
corollary~(\ref{c2}), corollary~(\ref{c3}), corollary~(\ref{c4}), and corollary~(\ref{c5}).

\begin{corollary}[\bf Product rule type-I ]
\label{c2}
Let $\alpha,\beta : [a,b] \times [a,b]\longrightarrow (0,\infty)$, $a,c\in \mathbb{R}$, $t>a$, and $t>c$. Then
\begin{equation*}
\bigg({}_{a}^{RL}I_{t}^{\alpha(.,.)}(fg)(t)\bigg) \bigg({}_{c}^{RL}I_{t}^{\beta(.,.)}(1) \bigg)+\bigg({}_{a}^{RL}I_{t}^{\alpha(.,.)}(1)\bigg)\bigg( {}_{c}^{RL}I_{t}^{\beta(.,.)}(fg)(t)\bigg)
\end{equation*}

\begin{equation}
\label{4.8}
=\bigg( {}_{a}^{RL}I_{t}^{\alpha(.,.)}g(t)\bigg)\bigg( {}_{c}^{RL}I_{t}^{\beta(.,.)}f(t)\bigg)+\bigg( {}_{a}^{RL}I_{t}^{\alpha(.,.)}f(t)\bigg)\bigg( {}_{c}^{RL}I_{t}^{\beta(.,.)}g(t)\bigg)
\end{equation}
\end{corollary}

\begin{proof}
From Theorem~(\ref{4.4}), equation~(\ref{4.3}). Letting $s=t$ completes the proof.
\end{proof}

\begin{corollary}[\bf Product rule type-II]
\label{c3}
Let $\alpha,\beta : [a,b] \times [a,b]\longrightarrow (0,\infty)$, $a\in \mathbb{R}$, $t>a$. Then
\begin{equation*}
\bigg({}_{a}^{RL}I_{t}^{\alpha(.,.)}(fg)(t)\bigg) \bigg({}_{a}^{RL}I_{t}^{\beta(.,.)}(1) \bigg)+\bigg({}_{a}^{RL}I_{t}^{\alpha(.,.)}(1)\bigg)\bigg( {}_{a}^{RL}I_{t}^{\beta(.,.)}(fg)(t)\bigg)
\end{equation*}

\begin{equation}
\label{4.9}
=\bigg( {}_{a}^{RL}I_{t}^{\alpha(.,.)}g(t)\bigg)\bigg( {}_{a}^{RL}I_{t}^{\beta(.,.)}f(t)\bigg)+\bigg( {}_{a}^{RL}I_{t}^{\alpha(.,.)}f(t)\bigg)\bigg( {}_{a}^{RL}I_{t}^{\beta(.,.)}g(t)\bigg)
\end{equation}
\end{corollary}

\begin{proof}
From Theorem~(\ref{4.4}), equation~(\ref{4.3}). Letting $s=t$ and $a=c$ completes the proof.
\end{proof}

\begin{corollary}[\bf Product rule type-III]
\label{c4}
Let $\alpha: [a,b] \times [a,b]\longrightarrow (0,\infty)$, $a\in \mathbb{R}$, $t>a$. Then 
\begin{equation}
\label{4.10}
\bigg({}_{a}^{RL}I_{t}^{\alpha(.,.)}(fg)(t)\bigg)= \bigg({}_{a}^{RL}I_{t}^{\alpha(.,.)}(1) \bigg)^{-1}
\bigg( {}_{a}^{RL}I_{t}^{\alpha(.,.)}g(t)\bigg)\bigg( {}_{a}^{RL}I_{t}^{\alpha(.,.)}f(t)\bigg)
\end{equation}
\end{corollary}

\begin{proof}
From Theorem~(\ref{4.4}), equation~(\ref{4.3}). Letting $s=t$, $a=c$, and $\alpha(.,.)=\beta(.,.)$ completes the proof.
\end{proof}

\begin{corollary}[\bf Product rule type-IV]
\label{c5}
Let $\alpha: [a,b] \times [a,b]\longrightarrow (0,\infty)$, $a\in \mathbb{R}$, $t>a$. Then 
\begin{equation}
\label{4.11}
\bigg({}_{a}^{RL}I_{t}^{\alpha(.,.)}f^2(t)\bigg)= \bigg({}_{a}^{RL}I_{t}^{\alpha(.,.)}(1) \bigg)^{-1}
\bigg( {}_{a}^{RL}I_{t}^{\alpha(.,.)}f(t)\bigg)^2
\end{equation}
\end{corollary}

\begin{proof}
From Theorem~(\ref{4.4}), equation~(\ref{4.3}). Letting $s=t$, $a=c$, and $\alpha(.,.)=\beta(.,.)$ and $f=g$ completes the proof.
\end{proof}
\begin{remark}
Quotient rule type-I, quotient rule type-II, quotient rule type-III, and quotient rule type-IV formulas is the same as product rule types, that is, from equation~(\ref{4.8}), equation~(\ref{4.9}), equation~(\ref{4.10}), and equation~(\ref{4.11}) respectively by letting $g=1/h$ such that $h$ is non zero.
 \end{remark}

\begin{theorem}
\label{i}
Let $\alpha : [a,b] \times [a,b]\longrightarrow (0,\infty)$, $a\in \mathbb{R}$, $t>a$, $n \in \mathbb{N}$. Then for function $f^n$ the following equality holds
\begin{equation}
\label{i.0}
\bigg({}_{a}^{RL}I_{t}^{\alpha(.,.)}f^n(t)\bigg)= \bigg({}_{a}^{RL}I_{t}^{\alpha(.,.)}(1) \bigg)^{-(n-1)}
\bigg( {}_{a}^{RL}I_{t}^{\alpha(.,.)}f(t)\bigg)^n
\end{equation}
\end{theorem}

\begin{proof}
Use mathematical induction. For $n=2$, equation~(\ref{i.0}) becomes product rule type-Iv. Now, assume that equation~(\ref{i.0}) is true for $n=k$. Let us show that equation~(\ref{i.0}) also holds for $n=k+1$, we have,
\begin{equation}
\label{i.1}
\bigg({}_{a}^{RL}I_{t}^{\alpha(.,.)}f^{k+1}(t)\bigg)=\bigg({}_{a}^{RL}I_{t}^{\alpha(.,.)}f^{k}(t)f(t)\bigg)
\end{equation}
now, use product rule type-III for the right-hand side of equation~(\ref{i.1}). Then we have,

\begin{equation}
\label{i.2}
\begin{split}
&\bigg({}_{a}^{RL}I_{t}^{\alpha(.,.)}f^{k+1}(t)\bigg)\\
&\quad=\bigg({}_{a}^{RL}I_{t}^{\alpha(.,.)}f^{k}(t)f(t)\bigg)\\
&\quad=\bigg({}_{a}^{RL}I_{t}^{\alpha(.,.)}(1)\bigg)^{-1}\bigg({}_{a}^{RL}I_{t}^{\alpha(.,.)}f(t)\bigg)\bigg({}_{a}^{RL}I_{t}^{\alpha(.,.)}f^{k}(t)\bigg)
\end{split}
\end{equation}
now, using our assumption for $n=k$ is true, equation~(\ref{i.2}) becomes,

\begin{equation*}
\begin{split}
&\bigg({}_{a}^{RL}I_{t}^{\alpha(.,.)}f^{k+1}(t)\bigg)\\
&\quad=\bigg({}_{a}^{RL}I_{t}^{\alpha(.,.)}f^{k}(t)f(t)\bigg)\\
&\quad=\bigg({}_{a}^{RL}I_{t}^{\alpha(.,.)}(1)\bigg)^{-1}\bigg({}_{a}^{RL}I_{t}^{\alpha(.,.)}f(t)\bigg)\bigg({}_{a}^{RL}I_{t}^{\alpha(.,.)}f^{k}(t)\bigg)\\
&\quad=\bigg({}_{a}^{RL}I_{t}^{\alpha(.,.)}(1)\bigg)^{-1}\bigg({}_{a}^{RL}I_{t}^{\alpha(.,.)}f(t)\bigg) \bigg({}_{a}^{RL}I_{t}^{\alpha(.,.)}(1) \bigg)^{-(k-1)}
\bigg( {}_{a}^{RL}I_{t}^{\alpha(.,.)}f(t)\bigg)^k\\
&\quad=\bigg({}_{a}^{RL}I_{t}^{\alpha(.,.)}(1) \bigg)^{-k}
\bigg( {}_{a}^{RL}I_{t}^{\alpha(.,.)}f(t)\bigg)^{k+1}.
\end{split}
\end{equation*}This completes the proof.

\end{proof}
\justify
For the Reimann-Liouville variable-order fractional integral operator, the following Theorem~(\ref{4.15}) established chain rule type-I and from the consequence of this Theorem we can obtain Chain rule type-II, chain rule type-III and chain rule type-IV. 
\begin{theorem}
\label{4.15}
Let $\alpha,\beta : [a,b] \times [a,b]\longrightarrow (0,\infty)$, $a,c\in \mathbb{R}$, $t>a$, $g(f)=(g\circ f)(x)$, where $f:=f(x)$ and for $f(t)>c$. Then we have

\begin{equation}
\label{4.12}
\bigg({}_{a}^{RL}I_{t}^{\alpha(.,.)}(g\circ f)(t)\bigg)
=\bigg({}_{c}^{RL}I_{f(t)}^{\beta(.,.)} g(f(t))\bigg) \frac{\bigg( {}_{a}^{RL}I_{t}^{\alpha(.,.)}(1)\bigg)}{\bigg({}_{c}^{RL}I_{f(t)}^{\beta(.,.)}(1)\bigg)}
\end{equation}
\end{theorem}

\begin{proof}This Theorem can be proved in two different approachs.\\
{\bf Method-I:  }
Using Riemann-Liouville variable-order fractional integral definition, we've 
\begin{equation*}
\begin{split}
&\bigg({}_{c}^{RL}I_{s}^{\beta(.,.)}\bigg( {}_{a}^{RL}I_{t}^{\alpha(.,.)}f(t)g(s)\bigg)\bigg)\\
&\quad=\bigg({}_{c}^{RL}I_{s}^{\beta(.,.)}\bigg( g(s)\bigg({}_{a}^{RL}I_{t}^{\alpha(.,.)}f(t)\bigg)\bigg)
\\
&\quad=\bigg( {}_{a}^{RL}I_{t}^{\alpha(.,.)}f(t)\bigg)\bigg({}_{c}^{RL}I_{s}^{\beta(.,.)} g(s)\bigg)
\end{split}
\end{equation*}
which implies 

\begin{equation}
\label{4.13}
\bigg({}_{c}^{RL}I_{s}^{\beta(.,.)}\bigg( {}_{a}^{RL}I_{t}^{\alpha(.,.)}f(t)g(s)\bigg)\bigg)=\bigg( {}_{a}^{RL}I_{t}^{\alpha(.,.)}f(t)\bigg)\bigg({}_{c}^{RL}I_{s}^{\beta(.,.)} g(s)\bigg)
\end{equation}

now suppose $s=f(t)$, then equation~(\ref{4.13}) becomes
\begin{equation}
\label{4.14}
\bigg({}_{c}^{RL}I_{f(t)}^{\beta(.,.)}\bigg( {}_{a}^{RL}I_{t}^{\alpha(.,.)}f(t)(g\circ f)(t)\bigg)\bigg)=\bigg( {}_{a}^{RL}I_{t}^{\alpha(.,.)}f(t)\bigg)\bigg({}_{c}^{RL}I_{f(t)}^{\beta(.,.)} g(f(t))\bigg) 
\end{equation}
use product rule type-III for the left-hand side of equation~(\ref{4.14}), that is,

\begin{equation*}
\begin{split}
&{}_{c}^{RL}I_{f(t)}^{\beta(.,.)}\bigg[\bigg( {}_{a}^{RL}I_{t}^{\alpha(.,.)}(1)\bigg)^{-1}\bigg({}_{a}^{RL}I_{t}^{\alpha(.,.)}f(t)\bigg)\bigg({}_{a}^{RL}I_{t}^{\alpha(.,.)}(g\circ f)(t)\bigg)\bigg]
\\
&\quad=\bigg({}_{c}^{RL}I_{f(t)}^{\beta(.,.)}\bigg( {}_{a}^{RL}I_{t}^{\alpha(.,.)}f(t)(g\circ f)(t)\bigg)\bigg)
\\
&\quad=\bigg( {}_{a}^{RL}I_{t}^{\alpha(.,.)}f(t)\bigg)\bigg({}_{c}^{RL}I_{f(t)}^{\beta(.,.)} g(f(t))\bigg) 
\end{split}
\end{equation*}

which means

\begin{equation*}
\begin{split}
&\bigg( {}_{a}^{RL}I_{t}^{\alpha(.,.)}(1)\bigg)^{-1}\bigg({}_{a}^{RL}I_{t}^{\alpha(.,.)}f(t)\bigg)\bigg({}_{a}^{RL}I_{t}^{\alpha(.,.)}(g\circ f)(t)\bigg)\bigg({}_{c}^{RL}I_{f(t)}^{\beta(.,.)}(1)\bigg)
\\
&\quad={}_{c}^{RL}I_{f(t)}^{\beta(.,.)}\bigg[\bigg( {}_{a}^{RL}I_{t}^{\alpha(.,.)}(1)\bigg)^{-1}\bigg({}_{a}^{RL}I_{t}^{\alpha(.,.)}f(t)\bigg)\bigg({}_{a}^{RL}I_{t}^{\alpha(.,.)}(g\circ f)(t)\bigg)\bigg]
\\
&\quad=\bigg({}_{c}^{RL}I_{f(t)}^{\beta(.,.)}\bigg( {}_{a}^{RL}I_{t}^{\alpha(.,.)}f(t)(g\circ f)(t)\bigg)\bigg)
\\
&\quad=\bigg( {}_{a}^{RL}I_{t}^{\alpha(.,.)}f(t)\bigg)\bigg({}_{c}^{RL}I_{f(t)}^{\beta(.,.)} g(f(t))\bigg) 
\end{split}
\end{equation*}

this implies

\begin{equation*}
\begin{split}
&\bigg( {}_{a}^{RL}I_{t}^{\alpha(.,.)}(1)\bigg)^{-1}\bigg({}_{a}^{RL}I_{t}^{\alpha(.,.)}f(t)\bigg)\bigg({}_{a}^{RL}I_{t}^{\alpha(.,.)}(g\circ f)(t)\bigg)\bigg({}_{c}^{RL}I_{f(t)}^{\beta(.,.)}(1)\bigg)
\\
&\quad=\bigg( {}_{a}^{RL}I_{t}^{\alpha(.,.)}f(t)\bigg)\bigg({}_{c}^{RL}I_{f(t)}^{\beta(.,.)} g(f(t))\bigg) 
\end{split}
\end{equation*}

this implies

\begin{equation*}
\bigg({}_{a}^{RL}I_{t}^{\alpha(.,.)}(g\circ f)(t)\bigg)
=\bigg({}_{c}^{RL}I_{f(t)}^{\beta(.,.)} g(f(t))\bigg) \frac{\bigg( {}_{a}^{RL}I_{t}^{\alpha(.,.)}(1)\bigg)}{\bigg({}_{c}^{RL}I_{f(t)}^{\beta(.,.)}(1)\bigg)}
\end{equation*}
{\bf Method-II: }Let $g(f)=(g\circ f)(x)$, where $f:=f(x)$ and then
multiplying this equation by $(t-x)^{\alpha(t,x)-1}/\Gamma(\alpha(t,x))$ and integrate with respect to $x$ from $a$ to $t$, we get,
\begin{equation*}
\int_{a}^{t}\frac{(t-x)^{\alpha(t,x)-1}}{\Gamma(\alpha(t,x))} g(f)dx=\int_{a}^{t}\frac{(t-x)^{\alpha(t,x)-1}}{\Gamma(\alpha(t,x))}(g\circ f)(x)dx
\end{equation*}

which means

\begin{equation}
\label{4.16}
g(f)\bigg({}_{a}^{RL}I_{t}^{\alpha(.,.)}(1)\bigg)={}_{a}^{RL}I_{t}^{\alpha(.,.)}(g\circ f)(t)
\end{equation}

multiply equation~(\ref{4.16}) by $(f(t)-f(x))^{\beta(f(t),f(x))-1}/\Gamma(\beta(f(t),f(x)))$ and integrate with respect to $f(x)$ from $c$ to $f(t)$, that is,
\begin{equation*}
\begin{split}
\int_{c}^{f(t)}\frac{(f(t)-f(x))^{\beta(f(t),f(x))-1}}{\Gamma(\beta(f(t),f(x)))} \bigg({}_{a}^{RL}I_{t}^{\alpha(.,.)}(1)\bigg)g(f(x))df(x)
\\[3mm]
=\int_{c}^{f(t)}\frac{(f(t)-f(x))^{\beta(f(t),f(x))-1}}{\Gamma(\beta(f(t),f(x)))}{}_{a}^{RL}I_{t}^{\alpha(.,.)}(g\circ f)(t)df(x)
\end{split}
\end{equation*}

which means

\begin{equation*}
\bigg({}_{a}^{RL}I_{t}^{\alpha(.,.)}(1)\bigg)\bigg({}_{c}^{RL}I_{f(t)}^{\beta(.,.)} g(f(t))\bigg)=\bigg({}_{a}^{RL}I_{t}^{\alpha(.,.)}(g\circ f)(t)\bigg)\bigg({}_{c}^{RL}I_{f(t)}^{\beta(.,.)}(1)\bigg)
\end{equation*}

this implies 

\begin{equation*}
\bigg({}_{a}^{RL}I_{t}^{\alpha(.,.)}(g\circ f)(t)\bigg)
=\bigg({}_{c}^{RL}I_{f(t)}^{\beta(.,.)} g(f(t))\bigg) \frac{\bigg( {}_{a}^{RL}I_{t}^{\alpha(.,.)}(1)\bigg)}{\bigg({}_{c}^{RL}I_{f(t)}^{\beta(.,.)}(1)\bigg)}
\end{equation*}
\end{proof}

From Theorem~(\ref{4.15}), we established the following corollary~(\ref{c1.0}), corollary~(\ref{c1.1}), and corollary~(\ref{c1.2}).

\begin{corollary}[\bf Chain rule type-II]
\label{c1.0}
Let $\alpha: [a,b] \times [a,b]\longrightarrow (0,\infty)$, $a,c\in \mathbb{R}$, $t>a$, $g(f)=(g\circ f)(x)$, where $f:=f(x)$ and for $f(t)>c$. Then we have,
\begin{equation}
\label{4.17}
\bigg({}_{a}^{RL}I_{t}^{\alpha(.,.)}(g\circ f)(t)\bigg)
=\bigg({}_{c}^{RL}I_{f(t)}^{\alpha(.,.)} g(f(t))\bigg) \frac{\bigg( {}_{a}^{RL}I_{t}^{\alpha(.,.)}(1)\bigg)}{\bigg({}_{c}^{RL}I_{f(t)}^{\alpha(.,.)}(1)\bigg)}
\end{equation}
\end{corollary}

\begin{proof}
From Theorem~(\ref{4.15}), equation~(\ref{4.12}). Letting $\alpha=\beta$ completes the proof.
\end{proof}

\begin{corollary}[\bf Chain rule type-III]
\label{c1.1}
Let $\alpha: [a,b] \times [a,b]\longrightarrow (0,\infty)$, $a\in \mathbb{R}$, $t>a$, $g(f)=(g\circ f)(x)$, where $f:=f(x)$ and for $f(t)>c$. Then we have,
\begin{equation}
\label{4.17}
\bigg({}_{a}^{RL}I_{t}^{\alpha(.,.)}(g\circ f)(t)\bigg)
=\bigg({}_{a}^{RL}I_{f(t)}^{\alpha(.,.)} g(f(t))\bigg) \frac{\bigg( {}_{a}^{RL}I_{t}^{\alpha(.,.)}(1)\bigg)}{\bigg({}_{a}^{RL}I_{f(t)}^{\alpha(.,.)}(1)\bigg)}
\end{equation}
\end{corollary}

\begin{proof}
From Theorem~(\ref{4.15}), equation~(\ref{4.12}). Letting $\alpha=\beta$ and $a=c$ completes the proof.
\end{proof}

\begin{corollary}[\bf Chain rule type-IV]
\label{c1.2}
Let $\alpha: [a,b] \times [a,b]\longrightarrow (0,\infty)$, $a\in \mathbb{R}$, $t>a$, $g(f)=(g\circ f)(x)$, where $f:=f(x)$ and for $f(t)>c$. Then we have,
\begin{equation}
\label{4.18}
\bigg({}_{a}^{RL}I_{t}^{\alpha(.,.)}(f\circ f)(t)\bigg)
=\bigg({}_{a}^{RL}I_{f(t)}^{\alpha(.,.)} f(f(t))\bigg) \frac{\bigg( {}_{a}^{RL}I_{t}^{\alpha(.,.)}(1)\bigg)}{\bigg({}_{a}^{RL}I_{f(t)}^{\alpha(.,.)}(1)\bigg)}
\end{equation}
\end{corollary}

\begin{proof}
From Theorem~(\ref{4.15}), equation~(\ref{4.12}). Letting $\alpha=\beta$, $a=c$ and $f=g$ completes the proof.
\end{proof}

\justify
In the following Theorem~(\ref{4.19}), equation~(\ref{4.23}) mentions the relationship   between variable-order Riemann-Liouville integrals of addition, subtraction and product of two functions with respect to two different variables beautifully. The consequences of this theorem becomes more beautifull.
\begin{theorem}
\label{4.19}
Let $\alpha,\beta : [a,b] \times [a,b]\longrightarrow (0,\infty)$, $a,c\in \mathbb{R}$, $t>a,s>c$. Then for functions $f$ and $g$: 

\begin{equation}
\label{4.23}
\begin{split}
&{}_{c}^{RL}I_{s}^{\beta(.,.)}\bigg({}_{a}^{RL}I_{t}^{\alpha(.,.)}(f(t)-f(s))(g(t)-g(s))\bigg) \\&=\bigg({}_{a}^{RL}I_{t}^{\alpha(.,.)}(f(t)g(t))\bigg)\bigg({}_{c}^{RL}I_{s}^{\beta(.,.)}(1)\bigg)+\bigg({}_{a}^{RL}I_{t}^{\alpha(.,.)}(1)\bigg)\bigg({}_{c}^{RL}I_{s}^{\beta(.,.)}f(s)g(s)\bigg)
\\&\quad+\frac{1}{2}\bigg[\bigg({}_{a}^{RL}I_{t}^{\alpha(.,.)}(f(t)-g(t))\bigg)\bigg({}_{c}^{RL}I_{s}^{\beta(.,.)}(f(s)-g(s))\bigg)
\\&\quad-\bigg({}_{a}^{RL}I_{t}^{\alpha(.,.)}(f(t)+g(t))\bigg)\bigg({}_{c}^{RL}I_{s}^{\beta(.,.)}(f(s)+g(s)) \bigg)
\bigg]
\end{split}
\end{equation}
\end{theorem}

\begin{proof}
Since,
\begin{equation}
\label{4.20}
\begin{split}
&(f(t)-f(s))(g(t)-g(s))\\
&\quad=f(t)g(t)+f(s)g(s)
\\&\qquad+\frac{1}{2}\bigg[\bigg(f(t)-g(t)\bigg)\bigg(f(s)-g(s)\bigg)\\
&\qquad-\bigg(f(t)+g(t)\bigg)\bigg(f(s)+g(s)\bigg)
\bigg]
\end{split}
\end{equation}
applying the operator ${}_{a}^{RL}I_{t}^{\alpha(.,.)}$ on equation~(\ref{4.20}) and use linearity property, we have,
\begin{equation}
\label{4.21}
\begin{split}
&{}_{a}^{RL}I_{t}^{\alpha(.,.)}\bigg((f(t)-f(s))(g(t)-g(s))\bigg)\\
&\quad={}_{a}^{RL}I_{t}^{\alpha(.,.)}\bigg(f(t)g(t)\bigg)+{}_{a}^{RL}I_{t}^{\alpha(.,.)}\bigg(f(s)g(s)\bigg)
\\&\qquad+\frac{1}{2}\bigg[{}_{a}^{RL}I_{t}^{\alpha(.,.)}\bigg(f(t)-g(t)\bigg)\bigg(f(s)-g(s)\bigg)\\
&\qquad-{}_{a}^{RL}I_{t}^{\alpha(.,.)}\bigg(f(t)+g(t)\bigg)\bigg(f(s)+g(s)\bigg)
\bigg]
\end{split}
\end{equation}
which means using linearity property,
\begin{equation}
\label{4.22}
\begin{split}
&{}_{a}^{RL}I_{t}^{\alpha(.,.)}\bigg((f(t)-f(s))(g(t)-g(s))\bigg)\\
&\quad={}_{a}^{RL}I_{t}^{\alpha(.,.)}(f(t)g(t))+f(s)g(s)\bigg({}_{a}^{RL}I_{t}^{\alpha(.,.)}(1)\bigg)
\\&\qquad+\frac{1}{2}\bigg[(f(s)-g(s))\bigg({}_{a}^{RL}I_{t}^{\alpha(.,.)}(f(t)-g(t))\bigg)\\
&\qquad-(f(s)+g(s))\bigg({}_{a}^{RL}I_{t}^{\alpha(.,.)}(f(t)+g(t))\bigg)
\bigg]
\end{split}
\end{equation}
applying the operator ${}_{c}^{RL}I_{s}^{\beta(.,.)}$ on equation~(\ref{4.22}) and use linearity property, we get,
\begin{equation*}
\begin{split}
&{}_{c}^{RL}I_{s}^{\beta(.,.)}\bigg({}_{a}^{RL}I_{t}^{\alpha(.,.)}(f(t)-f(s))(g(t)-g(s))\bigg) \\&={}_{c}^{RL}I_{s}^{\beta(.,.)}\bigg({}_{a}^{RL}I_{t}^{\alpha(.,.)}(f(t)g(t))\bigg)+\bigg({}_{c}^{RL}I_{s}^{\beta(.,.)}f(s)g(s)\bigg({}_{a}^{RL}I_{t}^{\alpha(.,.)}(1)\bigg)\bigg)
\\&\quad+\frac{1}{2}\bigg[\bigg({}_{c}^{RL}I_{s}^{\beta(.,.)}(f(s)-g(s))\bigg({}_{a}^{RL}I_{t}^{\alpha(.,.)}(f(t)-g(t))\bigg)\bigg)\\
&\quad-\bigg({}_{c}^{RL}I_{s}^{\beta(.,.)}(f(s)+g(s))\bigg({}_{a}^{RL}I_{t}^{\alpha(.,.)}(f(t)+g(t))\bigg) \bigg)
\bigg]
\end{split}
\end{equation*}
which means
\begin{equation*}
\begin{split}
&{}_{c}^{RL}I_{s}^{\beta(.,.)}\bigg({}_{a}^{RL}I_{t}^{\alpha(.,.)}(f(t)-f(s))(g(t)-g(s))\bigg) \\&=\bigg({}_{a}^{RL}I_{t}^{\alpha(.,.)}(f(t)g(t))\bigg)\bigg({}_{c}^{RL}I_{s}^{\beta(.,.)}(1)\bigg)+\bigg({}_{a}^{RL}I_{t}^{\alpha(.,.)}(1)\bigg)\bigg({}_{c}^{RL}I_{s}^{\beta(.,.)}f(s)g(s)\bigg)
\\&\quad+\frac{1}{2}\bigg[\bigg({}_{a}^{RL}I_{t}^{\alpha(.,.)}(f(t)-g(t))\bigg)\bigg({}_{c}^{RL}I_{s}^{\beta(.,.)}(f(s)-g(s))\bigg)
\\&\quad-\bigg({}_{a}^{RL}I_{t}^{\alpha(.,.)}(f(t)+g(t))\bigg)\bigg({}_{c}^{RL}I_{s}^{\beta(.,.)}(f(s)+g(s)) \bigg)
\bigg]
\end{split}
\end{equation*}
\end{proof}

\begin{corollary}
Let $\alpha,\beta : [a,b] \times [a,b]\longrightarrow (0,\infty)$, $a,c\in \mathbb{R}$, $t>a,s>c$. Then for functions $f$ and $g$ the following equality holds
\begin{equation}
\label{4.24}
\begin{split}
&\bigg({}_{a}^{RL}I_{t}^{\alpha(.,.)}(1)\bigg)^{-1}\bigg({}_{c}^{RL}I_{s}^{\beta(.,.)}(1)\bigg)^{-1}\bigg({}_{c}^{RL}I_{s}^{\beta(.,.)}\bigg({}_{a}^{RL}I_{t}^{\alpha(.,.)}(f(t)-f(s))\bigg)\bigg)^2 
\\&=\bigg({}_{a}^{RL}I_{t}^{\alpha(.,.)}f(t)\bigg)^2\bigg({}_{a}^{RL}I_{t}^{\alpha(.,.)}(1)\bigg)^{-1}\bigg({}_{c}^{RL}I_{s}^{\beta(.,.)}(1)\bigg)\\
&\quad+\bigg({}_{a}^{RL}I_{t}^{\alpha(.,.)}(1)\bigg)\bigg({}_{c}^{RL}I_{s}^{\beta(.,.)}f(s)\bigg)^2\bigg({}_{c}^{RL}I_{s}^{\beta(.,.)}(1)\bigg)^{-1}\\
&\quad-2\bigg({}_{a}^{RL}I_{t}^{\alpha(.,.)}f(t)\bigg)\bigg({}_{c}^{RL}I_{s}^{\beta(.,.)}f(s) \bigg)
\end{split}
\end{equation}
\end{corollary}

\begin{proof}
From equation~(\ref{4.23}), let $f=g$ and use product rule type-IV.
\end{proof}
\begin{corollary}
Let $\alpha,\beta : [a,b] \times [a,b]\longrightarrow (0,\infty)$, $a,c\in \mathbb{R}$, $t>a,s>c$. Then for functions $f$ and $g$ the following equality holds

\begin{equation}
\label{4.25}
\begin{split}
&\bigg({}_{a}^{RL}I_{t}^{\alpha(.,.)}(f(t)g(t))\bigg)\bigg({}_{c}^{RL}I_{t}^{\beta(.,.)}(1)\bigg)+\bigg({}_{a}^{RL}I_{t}^{\alpha(.,.)}(1)\bigg)\bigg({}_{c}^{RL}I_{t}^{\beta(.,.)}f(t)g(t)\bigg)
\\&\quad+\frac{1}{2}\bigg[\bigg({}_{a}^{RL}I_{t}^{\alpha(.,.)}(f(t)-g(t))\bigg)\bigg({}_{c}^{RL}I_{t}^{\beta(.,.)}(f(t)-g(t))\bigg)
\\&\quad-\bigg({}_{a}^{RL}I_{t}^{\alpha(.,.)}(f(t)+g(t))\bigg)\bigg({}_{c}^{RL}I_{t}^{\beta(.,.)}(f(t)+g(t)) \bigg)
\bigg]=0
\end{split}
\end{equation}
\end{corollary}
\begin{proof}
From equation~(\ref{4.23}), letting $s=t$ completes the proof.
\end{proof}
\justify
The next Theorem~(\ref{4.26}) will show us how to operate with Riemann-Liouville variable-order fractional integral operator of the product of two functions with two-variable.
\begin{theorem}
\label{4.26}
Let $\alpha,\beta : [a,b] \times [a,b]\longrightarrow (0,\infty)$, $a,c\in \mathbb{R}$, $t>a,s>c$. Then for functions $F$ and $G$ the following equality holds
\begin{equation}
\label{4.27}
\begin{split}
&{}_{a}^{RL}I_{t}^{\alpha(.,.)}{}_{c}^{RL}I_{s}^{\beta(.,.)}F(t,s)G(t,s)\\
&\quad=\bigg({}_{c}^{RL}I_{s}^{\beta(.,.)}(1)\bigg)^{-1}\bigg({}_{a}^{RL}I_{t}^{\alpha(.,.)}(1)\bigg)^{-1}\bigg({}_{a}^{RL}I_{t}^{\alpha(.,.)}\bigg({}_{c}^{RL}I_{s}^{\beta(.,.)}F(t,s)\bigg)\bigg)\\
&\quad \times \bigg({}_{a}^{RL}I_{t}^{\alpha(.,.)}\bigg({}_{c}^{RL}I_{s}^{\beta(.,.)}G(t,s)\bigg)\bigg)
\end{split}
\end{equation}
\end{theorem}

\begin{proof}Applying product rule type-III repeatedly, that is, 
\begin{equation*}
\begin{split}
&{}_{a}^{RL}I_{t}^{\alpha(.,.)}{}_{c}^{RL}I_{s}^{\beta(.,.)}F(t,s)G(t,s)\\
&\quad={}_{a}^{RL}I_{t}^{\alpha(.,.)}\bigg({}_{c}^{RL}I_{s}^{\beta(.,.)}F(t,s)G(t,s)\bigg)\\
&\quad={}_{a}^{RL}I_{t}^{\alpha(.,.)}\bigg(\bigg({}_{c}^{RL}I_{s}^{\beta(.,.)}(1)\bigg)^{-1}\bigg({}_{c}^{RL}I_{s}^{\beta(.,.)}F(t,s)\bigg)\bigg({}_{c}^{RL}I_{s}^{\beta(.,.)}G(t,s)\bigg)\bigg)\\
&\quad=\bigg({}_{c}^{RL}I_{s}^{\beta(.,.)}(1)\bigg)^{-1}\bigg({}_{a}^{RL}I_{t}^{\alpha(.,.)}\bigg(\bigg({}_{c}^{RL}I_{s}^{\beta(.,.)}F(t,s)\bigg)\bigg({}_{c}^{RL}I_{s}^{\beta(.,.)}G(t,s)\bigg)\bigg)\\
&\quad=\bigg({}_{c}^{RL}I_{s}^{\beta(.,.)}(1)\bigg)^{-1}\bigg({}_{a}^{RL}I_{t}^{\alpha(.,.)}(1)\bigg)^{-1}\bigg({}_{a}^{RL}I_{t}^{\alpha(.,.)}\bigg({}_{c}^{RL}I_{s}^{\beta(.,.)}F(t,s)\bigg)\bigg)\\
&\quad \times \bigg({}_{a}^{RL}I_{t}^{\alpha(.,.)}\bigg({}_{c}^{RL}I_{s}^{\beta(.,.)}G(t,s)\bigg)\bigg)
\end{split}
\end{equation*}this implies
\begin{equation*}
\begin{split}
&{}_{a}^{RL}I_{t}^{\alpha(.,.)}{}_{c}^{RL}I_{s}^{\beta(.,.)}F(t,s)G(t,s)\\
&\quad=\bigg({}_{c}^{RL}I_{s}^{\beta(.,.)}(1)\bigg)^{-1}\bigg({}_{a}^{RL}I_{t}^{\alpha(.,.)}(1)\bigg)^{-1}\bigg({}_{a}^{RL}I_{t}^{\alpha(.,.)}\bigg({}_{c}^{RL}I_{s}^{\beta(.,.)}F(t,s)\bigg)\bigg)\\
&\quad \times \bigg({}_{a}^{RL}I_{t}^{\alpha(.,.)}\bigg({}_{c}^{RL}I_{s}^{\beta(.,.)}G(t,s)\bigg)\bigg)
\end{split}
\end{equation*}
\end{proof}

\begin{remark}
To find the product rule, quotient rule, and chain rule formulas for Riemann-Liouville variable-order fractional derivative operator, use definition~(\ref{2.3}), that is,
\begin{equation}
{}_{a}^{RL}D_{t}^{\alpha(.,.)}f(t)=\frac{d}{dt}\bigg({}_{a}^{RL}I_{t}^{1-\alpha(.,.)}f(t)\bigg)
\end{equation}where $\alpha : [a,b] \times [a,b]\longrightarrow (0,1)$, $a\in \mathbb{R}$ and $t>a$. For example, let's see the next Theorem~(\ref{4.28}) which is product rule type-III.
\end{remark}

\begin{theorem}
\label{4.28}
Let $\alpha: [a,b] \times [a,b]\longrightarrow (0,1)$, $a\in \mathbb{R}$, $t>a$. Then 
\begin{equation}
\label{4.29}
\begin{split}
&\bigg({}_{a}^{RL}D_{t}^{\alpha(.,.)}(fg)(t)\bigg)\\
&\quad= -\bigg({}_{a}^{RL}I_{t}^{1-\alpha(.,.)}(1) \bigg)^{-2}
\bigg( {}_{a}^{RL}I_{t}^{1-\alpha(.,.)}g(t)\bigg)\bigg( {}_{a}^{RL}I_{t}^{1-\alpha(.,.)}f(t)\bigg)\bigg({}_{a}^{RL}D_{t}^{\alpha(.,.)}(1) \bigg)\\
&\qquad+ \bigg({}_{a}^{RL}I_{t}^{1-\alpha(.,.)}(1) \bigg)^{-1}
\bigg( {}_{a}^{RL}I_{t}^{1-\alpha(.,.)}g(t)\bigg)\bigg( {}_{a}^{RL}D_{t}^{\alpha(.,.)}f(t)\bigg)\\
&\qquad +\bigg({}_{a}^{RL}I_{t}^{1-\alpha(.,.)}(1) \bigg)^{-1}
\bigg( {}_{a}^{RL}I_{t}^{1-\alpha(.,.)}g(t)\bigg)\bigg( {}_{a}^{RL}D_{t}^{\alpha(.,.)}f(t)\bigg)
\end{split}
\end{equation}
\end{theorem}

\begin{proof}
From definition~(\ref{2.3}), we have,
\begin{equation}
\label{4.30}
{}_{a}^{RL}D_{t}^{\alpha(.,.)}f(t)g(t)=\frac{d}{dt}\bigg({}_{a}^{RL}I_{t}^{1-\alpha(.,.)}f(t)g(t)\bigg)
\end{equation}now use product rule type-III for the right-hand side of equation~(\ref{4.30}), we have,
\begin{equation}
\label{4.31}
\begin{split}
{}_{a}^{RL}D_{t}^{\alpha(.,.)}f(t)g(t)&=\frac{d}{dt}\bigg({}_{a}^{RL}I_{t}^{1-\alpha(.,.)}f(t)g(t)\bigg)\\
&=\frac{d}{dt}\bigg(\bigg({}_{a}^{RL}I_{t}^{1-\alpha(.,.)}(1)\bigg)^{-1}\bigg( {}_{a}^{RL}I_{t}^{1-\alpha(.,.)}f(t)\bigg)\bigg({}_{a}^{RL}I_{t}^{1-\alpha(.,.)}g(t) \bigg)\bigg)
\end{split}
\end{equation}
now use Leibniz product Rule for the right-hand side of equation~(\ref{4.31}), we have,
\begin{equation*}
\begin{split}
{}_{a}^{RL}D_{t}^{\alpha(.,.)}f(t)g(t)&=\frac{d}{dt}\bigg({}_{a}^{RL}I_{t}^{1-\alpha(.,.)}f(t)g(t)\bigg)\\
&=\frac{d}{dt}\bigg(\bigg({}_{a}^{RL}I_{t}^{1-\alpha(.,.)}(1)\bigg)^{-1}\bigg( {}_{a}^{RL}I_{t}^{1-\alpha(.,.)}f(t)\bigg)\bigg({}_{a}^{RL}I_{t}^{1-\alpha(.,.)}g(t) \bigg)\bigg)\\
&=\bigg( {}_{a}^{RL}I_{t}^{1-\alpha(.,.)}f(t)\bigg)\bigg({}_{a}^{RL}I_{t}^{1-\alpha(.,.)}g(t) \bigg)\frac{d}{dt}\bigg({}_{a}^{RL}I_{t}^{1-\alpha(.,.)}(1)\bigg)^{-1}\\
&\quad+\bigg({}_{a}^{RL}I_{t}^{1-\alpha(.,.)}(1)\bigg)^{-1}\bigg({}_{a}^{RL}I_{t}^{1-\alpha(.,.)}g(t) \bigg)\frac{d}{dt}\bigg( {}_{a}^{RL}I_{t}^{1-\alpha(.,.)}f(t)\bigg)\\
&\quad+\bigg({}_{a}^{RL}I_{t}^{1-\alpha(.,.)}(1)\bigg)^{-1}\bigg( {}_{a}^{RL}I_{t}^{1-\alpha(.,.)}f(t)\bigg)\frac{d}{dt}\bigg({}_{a}^{RL}I_{t}^{1-\alpha(.,.)}g(t) \bigg)\\
&= -\bigg({}_{a}^{RL}I_{t}^{1-\alpha(.,.)}(1) \bigg)^{-2}
\bigg( {}_{a}^{RL}I_{t}^{1-\alpha(.,.)}g(t)\bigg)\bigg( {}_{a}^{RL}I_{t}^{1-\alpha(.,.)}f(t)\bigg)\bigg({}_{a}^{RL}D_{t}^{\alpha(.,.)}(1) \bigg)\\
&\quad+ \bigg({}_{a}^{RL}I_{t}^{1-\alpha(.,.)}(1) \bigg)^{-1}
\bigg( {}_{a}^{RL}I_{t}^{1-\alpha(.,.)}g(t)\bigg)\bigg( {}_{a}^{RL}D_{t}^{\alpha(.,.)}f(t)\bigg)\\
&\quad +\bigg({}_{a}^{RL}I_{t}^{1-\alpha(.,.)}(1) \bigg)^{-1}
\bigg( {}_{a}^{RL}I_{t}^{1-\alpha(.,.)}g(t)\bigg)\bigg( {}_{a}^{RL}D_{t}^{\alpha(.,.)}f(t)\bigg)
\end{split}
\end{equation*}
\end{proof}

\justify
{\bf Authors’ contributions: }
All authors worked jointly and all the authors read and approved the final manuscript.\\
{\bf Funding:} 
This research received no external funding.\\
{\bf Conflicts of Interest:} The authors declare no conflict of interest.

\end{document}